\documentclass[a4paper, 10pt]{amsart}
\usepackage{amsmath}
\usepackage{amssymb, color, hyperref}
\usepackage[mathscr]{eucal}

\def\doi#1{{\small\href{https://doi.org/#1}{\path{doi:#1}}}}
\def\arxiv#1{{\small\href{http://www.arxiv.org/abs/#1}{\path{arXiv:#1}}}}
\def\url#1{{\small\href{#1}{\path{#1}}}}

\theoremstyle{plain}
\newtheorem{theorem}{\bf Theorem}[section]
\newtheorem{proposition}[theorem]{\bf Proposition}
\newtheorem{lemma}[theorem]{\bf Lemma}
\newtheorem{corollary}[theorem]{\bf Corollary}

\theoremstyle{definition}

\newcommand{\N}{\mathbb N}
\newcommand{\Z}{\mathbb Z}

\newcommand{\Q}{\mathbb Q}

\newcommand{\red}{{\text{\rm red}}}
\renewcommand{\t}{\, | \,}

\newcommand{\fin}{\text{\rm fin}}

\numberwithin{equation}{section}


\thanks{This work was supported by the Austrian Science Fund FWF, Project P33499-N}

\begin{document}

\title{A realization result for systems of sets of lengths}

\author{Alfred Geroldinger and  Qinghai Zhong}

\address{University of Graz, NAWI Graz \\
Institute for Mathematics and Scientific Computing \\
Heinrichstra{\ss}e 36\\
8010 Graz, Austria}
\email{alfred.geroldinger@uni-graz.at,  qinghai.zhong@uni-graz.at}
\urladdr{https://imsc.uni-graz.at/geroldinger, https://imsc.uni-graz.at/zhong/}


\begin{abstract}
Let $\mathcal L^*$ be a family of finite subsets of $\N_0$ having the following properties.
\begin{itemize}
\item[(a)] $\{0\}, \{1\} \in \mathcal L^*$ and all other sets of $\mathcal L^*$ lie in $\N_{\ge 2}$.

\item[(b)] If $L_1, L_2 \in \mathcal L^*$, then the sumset $L_1 + L_2 = \{l_1+l_2 \colon l_1 \in L_1, l_2 \in L_2\} \in \mathcal L^*$.
\end{itemize}
We show that there is a Dedekind domain $D$ whose system of sets of lengths equals $\mathcal L^*$.
\end{abstract}

\maketitle


\section{Introduction} \label{1}

Let $D$ be a domain or a monoid and suppose that every non-zero  non-invertible element can be written as a product of irreducible elements. The existence of such a factorization follows, among others, from weak ideal theoretic conditions on $D$. In general, factorizations are not unique. Sets of lengths have received a lot of attention because of their usefulness in describing the non-uniqueness of factorizations. To fix notation, let $a \in D$ be a non-zero non-invertible element. If $a = u_1 \cdot \ldots \cdot u_k$, where $k \in \N$ and $u_1, \ldots, u_k$ are irreducible elements of $D$, then $k$ is called a factorization length and the set $\mathsf L (a) \subset \N$ of all possible factorization lengths denotes the set of lengths of $a$. If $a$ is irreducible, then $\mathsf L (a)= \{1\}$ and it is convenient to  set $\mathsf L (a) = \{0\}$ for invertible elements $a \in D$. Then
\[
\mathcal L (D) = \{ \mathsf L (a) \colon a \ \text{is a nonzero element of $D$} \}
\]
denotes the {\it system of sets of lengths} of $D$. In an overwhelming number of settings studied so far, sets of lengths are finite. In particular, if $D$ is a commutative Noetherian domain (more generally, a commutative Mori domain or a commutative Mori monoid), then sets of lengths are finite.

We discuss basic properties of $\mathcal L (D)$. To do so, let us suppose now that $D$ is a commutative integral domain. Then $D$ is factorial if and only if $D$ is a Krull domain with trivial class group, and in this case we have
\[
\mathcal L (D) = \big\{ \{k\} \colon k \in \N_0 \big\} \,.
\]
Domains with this property are called half-factorial and they found wide attention in the literature (see \cite{Ch-Co00,Co05a,Co-Sm11a,Pl-Sc05b} for surveys and recent contributions). Among others, Krull domains, whose class group has at most two elements, are half-factorial. Suppose that $D$ is not half-factorial. Then, there is $a \in D$ such that $|\mathsf L (a)| > 1$ and hence, for every $n \in \N$, the $n$-fold sumset $\mathsf L (a) + \ldots + \mathsf L (a) \subset \mathsf L (a^n)$. Thus, $|\mathsf L (a^n)|>n$ and so sets of lengths can be arbitrarily large. However, in Krull domains with finite class group, sets of lengths are well-structured. They are almost arithmetical multiprogressions with global bounds for all parameters. The same is true for various classes of domains and we refer to \cite[Section 4.7]{Ge-HK06a} for a survey. In contrast to this, there are domains where every finite subset of $\N_{\ge 2}$ occurs as a set of lengths. Krull domains with infinite class group, where each class contains a height-one prime ideal (which holds true for all cluster algebras that are Krull \cite{Ga-La-Sm19a}),  classes of integer-valued polynomials, and others have this property (\cite{Ka99a}, \cite[Theorem 7.4.1]{Ge-HK06a}, \cite{Fr13a, Fr-Na-Ri19a}, \cite[Theorem 3.6]{Go19a}).

A standing open problem is to understand which families $\mathcal L^*$ of finite subsets of the non-negative integers occur as the system of sets of lengths of a monoid or domain (\cite[Problem B]{Ge-Zh20a}). If $1_D \in D$ is the identity element, then (as we have already said)  $\mathsf L (1_D) = \{0\}$. Let $a, b \in D$. Then $a$  is irreducible if and only if $1 \in \mathsf L (a)$ if and only if $\mathsf L (a) = \{1\}$. If $a = u_1 \cdot \ldots \cdot u_k$ and $b = v_1 \cdot \ldots \cdot v_{\ell}$, where all $u_i$ and $v_j$ are irreducibles, then $ab= u_1 \cdot \ldots \cdot u_kv_1 \cdot \ldots \cdot v_{\ell}$, whence $\mathsf L (a) + \mathsf L (b) \subset \mathsf L (ab)$. Thus, if $\mathcal L^*$ is a system of sets of lengths, then it satisfies the following two properties.
\begin{itemize}
\item[(a)] $\{0\}, \{1\} \in \mathcal L^*$ and all other sets of $\mathcal L^*$ lie in $\N_{\ge 2}$.

\item[(b)] If $L_1, L_2 \in \mathcal L^*$, then $L_1 + L_2 \subset L_3$ for some $L_3 \in \mathcal L^*$.
\end{itemize}

In the present paper, we show a partial converse. Indeed, if a family $\mathcal L^*$ satisfies (a) and (b) and if, in addition, the sumset $L_1+L_2$ is not only contained in a set of $\mathcal L^*$ but is actually in $\mathcal L^*$, then $\mathcal L^*$ is a system of sets of lengths. We formulate the main result of the present paper.

\medskip
\begin{theorem} \label{1.1}
Let $\mathcal L^*$ be a family of finite subsets of $\N_0$ having the following properties.
\begin{itemize}
\item[(a)] $\{0\}, \{1\} \in \mathcal L^*$ and all other sets of $\mathcal L^*$ lie in $\N_{\ge 2}$.

\item[(b)] If $L_1, L_2 \in \mathcal L^*$, then the sumset $L_1 + L_2 \in \mathcal L^*$.
\end{itemize}
Then there is a Krull monoid $H$ such that $\mathcal L (H) = \mathcal L^*$. Moreover, there is a finitely generated monoid $H^*$ (equivalently, a finitely generated Krull monoid $H^*$)  with $\mathcal L (H^*) = \mathcal L ^*$ if and only if $\mathcal L^*$ has only finitely many indecomposable sets.
\end{theorem}

If $a, b \in D$, then (as mentioned above) $\mathsf L (a) + \mathsf L (b) \subset \mathsf L (ab)$ and, in general, the containment is strict. For Krull monoids having prime divisors in all classes, there is a characterization (in terms of the class group) when  Property (b) is satisfied (\cite{Ge-Sc19d}).  Examples from module theory, which satisfy Property (b), can be found in  \cite[Section 6]{Ba-Ge14b} and  non-commutative monoids, which satisfy Property (b), are given in \cite[Theorem 4.5]{Ge-Sc18b}.

Krull monoids allow a variety of realization theorems. Let $H$ be a  monoid and $H^{\times}$ its group of invertible elements. Then $H$ is Krull if and only if the associated reduced monoid $H/H^{\times}$ is Krull. We recall some realization theorems for  Krull monoids:
\begin{itemize}
\item[(i)] Every reduced Krull monoid is isomorphic to a monoid of zero-sum sequences over a subset of an abelian group (\cite[Theorem 2.5.8]{Ge-HK06a}).

\item[(ii)] Every reduced Krull monoid is isomorphic to a monoid of isomorphism classes of projective modules (\cite[Theorem 2.1]{Fa-Wi04}, \cite{Fa19a}).

\item[(iii)] If the torsion subgroup of $H^{\times}$ of a Krull monoid $H$ is isomorphic to a subgroup of $\Q/\Z$, then $H$ is isomorphic to an arithmetically closed submonoid of a Dedekind domain, which is a quadratic extension of a principal ideal domain (\cite[Theorem 4]{Ge-HK92a}).
\end{itemize}
Combining Theorem \ref{1.1} with the above mentioned realization results, we see that systems $\mathcal L^*$ can be realized as a system of sets of lengths  of Krull monoids, as given in (i) - (iii). Lemma \ref{3.2} shows how to obtain realization results for transfer Krull monoids.

In order to obtain a realization result for Krull domains, we need a further ingredient. It is well-known that
a domain  is a Krull domain if and only if its monoid of nonzero elements is a Krull monoid. However, not every Krull monoid stems from a domain (Krull domains satisfy the approximation property but Krull monoids do not do so in general). Nevertheless, by combining Theorem \ref{1.1} with  Claborn's Realization Theorem (\cite{Cl66}) for class groups of Dedekind domains, we infer that a system $\mathcal L^*$, as given in Theorem \ref{1.1}, even occurs as the system of sets of lengths of a Dedekind domain. We formulate this as a corollary.

\begin{corollary} \label{1.2}
Let $\mathcal L^*$ be a family of finite subsets of $\N_0$ satisfying the properties (a) and (b), of the statement of Theorem \ref{1.1}.
Then there is a Dedekind domain $D$ such that $\mathcal L (D) = \mathcal L^*$. Moreover, there is a Dedekind domain $D^*$ with $\mathcal L (D^*) = \mathcal L^*$ such that the number of classes of the divisor class group $\mathcal C (D^*)$ containing height-one prime ideals is finite if and only if $\mathcal L^*$ has only finitely many indecomposable sets.
\end{corollary}

We proceed as follows. In Section \ref{2}, we gather the required background on sets of lengths and on Krull monoids. The proofs of Theorem \ref{1.1} and of Corollary \ref{1.2} will be given in Section \ref{3}.

\smallskip
\section{Background on sets of lengths and on Krull monoids} \label{2}

For elements $a, b \in \Z$, we denote by $[a, b ] = \{ x \in \Z \colon a \le x \le b \}$ the discrete interval between $a$ and $b$. Let $L, L' \subset \Z$ be subsets. Then $L+L' = \{a + a' \colon a \in L, a' \in L' \}$ is their sumset. For $n \in \N_0$,  $nL = L + \ldots + L$ is the $n$-fold sumset and $n \cdot L = \{ n a \colon a \in L\}$ is the dilation of $L$ by $n$ (with the convention that  $nL = \{0\}$ if $n=0$).  We denote by $\Delta (L) \subset \N$ the set of (successive) distances of $L$. Thus $\Delta (L) = \{d\}$ if $L$ is an arithmetical progression with difference $d$.

Let $\mathcal L \subset P_{\fin} (\N_0)$ be a family of finite subsets of $\N_0$ with $\{0\} \in \mathcal L$. We say that $\mathcal L$ is {\it additively closed} if the sumset $L_1+L_2 \in \mathcal L$ for all $L_1, L_2 \in \mathcal L$. A set $L \in \mathcal L$ is {\it indecomposable} (in $\mathcal L$) if $L_1, L_2 \in \mathcal L$ and $L = L_1+L_2$ implies that $L_1 = \{0\}$ or $L_2 = \{0\}$. Thus $\mathcal L$ is additively closed if and only if $\mathcal L$ is a semigroup with set addition as operation and with identity element $\{0\}$.

We briefly gather some arithmetical concepts of semigroups. The arithmetic of (in general non-cancellative subsemigroups)  $\mathcal L \subset P_{\fin} (\N_0)$ is studied in \cite{Fa-Tr18a}. Since the present focus is on realization results by Krull monoids (which are cancellative), we restrict to cancellative semigroups. Our notation and terminology is consistent with \cite{Ge-HK06a}.

\smallskip
\noindent
{\bf Monoids.} By a {\it monoid}, we mean a commutative and cancellative semigroup with identity element. For a set $P$, we denote by $\mathcal F (P)$ the free abelian monoid with basis $P$. An element $a \in \mathcal F (P)$ will be written in the form
\[
a = \prod_{p \in P} p^{\mathsf v_p (a)} , \quad \text{where $\mathsf v_p (a) = 0$ \ for almost all $p \in P$,}
\]
and $|a| = \sum_{p \in P} \mathsf v_p (a) \in \N_0$ is the length of $a$.
Let $H$ be a monoid. We denote by $H^{\times}$ the group of invertible elements, by $\mathsf q (H)$ the quotient group of $H$, by $H_{\red} = H/H^{\times}$ the associated reduced monoid, by $\mathcal A (H)$ the set of atoms (irreducible elements) of $H$, by $\mathfrak X (H)$ the set of minimal prime $s$-ideals of $H$, and by
\[
\widehat H = \{ x \in \mathsf q (H) \colon \ \text{there is $c \in H$ such that $cx^n \in H$ for all $n \in \N$} \}
\]
the complete integral closure of $H$. We say that $H$ is completely integrally closed  if $H = \widehat H$. The monoid $H$ is a Krull monoid if it satisfies one of the following equivalent conditions (\cite[Theorem 2.4.8]{Ge-HK06a}, \cite[Chapter 22]{HK98}):
\begin{itemize}
\item $H$ is completely integrally closed and satisfies the ascending chain condition on divisorial ideals.

\item There is a free abelian monoid $F = \mathcal F (P)$ and a divisor theory $\partial \colon H \to F$.
\end{itemize}
Let $H$ be a Krull monoid and suppose that $H_{\red} \hookrightarrow F = \mathcal F (P)$ a divisor theory. Then $P$ is called the set of prime divisors of $F$ and
\[
\mathcal C (H) = \mathsf q (F)/\mathsf q (H_{\red} )
\]
is the (divisor) class group of $H$. It is isomorphic to the $v$-class group $\mathcal C_v (H)$, which is the group of fractional divisorial ideals modulo the set of fractional principal ideals. Then  $G_P = \{ [p] = p\mathsf q (H_{\red} \colon p \in P\}$ is the set of classes containing prime divisors (see \cite[Definition 2.4.9]{Ge-HK06a}). We need the following lemma (\cite[Theorem 2.7.14]{Ge-HK06a}).

\begin{lemma} \label{2.1}
For a reduced Krull monoid $H$ with divisor theory $H \hookrightarrow \mathcal F (P)$ the following statements are equivalent.
\begin{enumerate}
\item[(a)] $H$ is finitely generated.

\item[(b)] The set of prime divisors $P$ is finite.

\item[(c)] $\mathfrak X (H)$ is finite.
\end{enumerate}
\end{lemma}

A (commutative integral) domain $D$ is a Krull domain if and only if its multiplicative monoid $D^{\bullet} : = D \setminus \{0\}$ is a Krull monoid (this was first proved independently by Krause \cite{Kr89} and Wauters \cite{Wa84a}, before it was generalized to non-commutative rings and to commutative rings with zero-divisors). If this holds, then the class group $\mathcal C (D)$ of the domain and the class group of the monoid $\mathcal C (D^{\bullet})$ coincide.

Let $G$ be an additive abelian group and $G_0 \subset G$ a subset. If $S= g_1 \cdot \ldots \cdot g_{\ell} \in \mathcal F (G_0)$, then $\sigma (S) = g_1 + \ldots + g_{\ell}$ is the sum of $S$. The set
\[
\mathcal B (G_0) = \{ S \in \mathcal F (G_0) \colon \sigma (S)=0 \} \subset \mathcal F (G_0)
\]
is a submonoid of $\mathcal F (G_0)$, called the {\it monoid of zero-sum sequences} over $G_0$, and it is a Krull monoid.

\smallskip
\noindent
{\bf Arithmetic of Monoids.} Let $H$ be a monoid. Then the free abelian monoid $\mathsf Z (H) = \mathcal F ( \mathcal A (H_{\red} ))$ is the factorization monoid of $H$ and the canonical epimorphism $\pi \colon \mathsf Z (H) \to H_{\red}$ is the factorization homomorphism.  If $\pi$ is surjective, then $H$ is called {\it atomic}. For an element $a \in H$,
\begin{itemize}
\item $\mathsf Z (a) = \pi^{-1} (aH^{\times}) \subset \mathsf Z (H)$ is the {\it set of factorizations} of $a$, and

\item $\mathsf L_H (a) = \mathsf L (a) = \{ |z| \colon z \in \mathsf Z (a) \} \subset \N_0$ is the {\it set of lengths} of $a$.
\end{itemize}
We say that $a$ has unique factorization if $|\mathsf Z (a)|=1$ and that $H$ is factorial if $|\mathsf Z (a)|=1$ for all $a \in H$. Then
\[
\begin{aligned}
\mathcal L (H) & = \{ \mathsf L (a) \colon a \in H \} \quad \text{is the {\it system of sets of lengths} of $H$, and} \\
\Delta (H) & = \bigcup_{L \in \mathcal L (H)} \Delta (L) \ \subset \N \quad \text{is the {\it set of distances} of $H$} \,.
\end{aligned}
\]
If $\Delta (H) \ne \emptyset$, then $\min \Delta (H) = \gcd \Delta (H)$.
If $H$ is a Krull monoid with divisor theory $\partial \colon H \to \mathcal F (P)$ and $a \in H$ with $\partial (a) = p_1 \cdot \ldots \cdot p_{\ell}$, where $p_1, \ldots, p_{\ell} \in P$, then $\sup \mathsf L_H (a) \le \ell$. Thus, all sets of lengths of $H$ are finite.
Let $z, z' \in \mathsf Z (H)$ be two factorizations. Then we can write them in the form
\[
z = u_1 \cdot \ldots \cdot u_{\ell}v_1 \cdot \ldots \cdot v_m \quad \text{and} \quad z' = u_1 \cdot \ldots \cdot u_{\ell} w_1 \cdot \ldots \cdot w_n \,,
\]
where all $u_i, v_j, w_k \in \mathcal A (H_{\red})$ and the $v_j$ and $w_k$ are pairwise distinct, and we call $\mathsf d (z,z') = \max \{m, n \} \in \N_0$  the {\it distance} of between $z$ and $z'$. The {\it catenary degree} $\mathsf c (a) \in \N_0 \cup \{\infty\}$ of an element $a \in H$ is the smallest $N \in \N_0 \cup \{\infty\}$ such that for each two factorizations $z, z' \in \mathsf Z (a)$ there are $z=z_0, z_1, \ldots, z_s = z'$ in $\mathsf Z (a)$ such that $\mathsf d (z_{i-1}, z_i) \le N$ for all $i \in [1, s]$. Then
\[
\mathsf c (H) = \sup \{\mathsf c (a) \colon a \in H \} \in \N_0 \cup \{\infty\}
\]
denotes the {\it catenary degree} of $H$. Note that $H$ is factorial if and only if $\mathsf c (H)=0$. If $\Delta (H) \ne \emptyset$, then
\[
2 + \sup \Delta (H) \le \mathsf c (H) \,,
\]
but there are Dedekind domains $D$ with $\Delta (D)= \emptyset$ and $\mathsf c (D) = \infty$.

\smallskip
\section{Proof of Theorem \ref{1.1} and of Corollary \ref{1.2}} \label{3}
\smallskip

\begin{proposition} \label{3.1}
Let $r \in \N$ and $L = \{k_1, \ldots, k_r\}  \subset \N_{\ge 2}$. Then there exists a reduced  finitely generated Krull monoid $H$ with $\mathcal A (H) = \{ u_{i,j} \colon j \in [1, k_i], i \in [1,r] \}$ and which has the following properties.
\begin{itemize}
\item[(a)] $u_{1,1} \cdot \ldots \cdot u_{1, k_1}  = \ldots  = u_{r,1} \cdot \ldots \cdot u_{r, k_r}$.

\item[(b)] Every $b \in H \setminus u_{1,1} \cdot \ldots \cdot u_{1, k_1}H$ has unique factorization.

\item[(c)] For every $a \in H$, there is a unique $n \in \N_0$ and a unique $b \in H \setminus u_{1,1} \cdot \ldots \cdot u_{1,k_1}H$ such that $a = (u_{1,1} \cdot \ldots \cdot u_{1,k_1})^n b$ and $\mathsf L (a) = n L + \mathsf L (b)$.
\end{itemize}
Moreover,  the catenary degree $\mathsf c (H)$ of $H$ is $0$ for $r=1$, is $k_r$ for $r > 1$, and
\[
\mathcal L (H) = \big\{ \{0\}, \{1\} \big\} \cup \{ y+ n L \colon  \ y, n  \in \N_0 \} \,.
\]
\end{proposition}

\begin{proof}
The statement on $\mathcal L (H)$ follows immediately from Property (c). In order to show the existence of a Krull monoid with the given properties we proceed by induction on $r$. If $r=1$, then the free abelian monoid $H_1$ with basis $\{u_{1,1}, \ldots, u_{1, k_1} \}$ has all required properties. In particular, $\mathsf c (H)=0$ and
\[
\mathcal L (H) = \big\{ \{y\} \colon y \in \N_0 \big\} \,.
\]
Let $r >1$ and suppose that there is a monoid $H_{r-1}$ with all the wanted properties. Let $F$ be the free abelian monoid with basis $\{ u_{r,1} , \ldots ,u_{r, k_r-1} \}$. Let
\[
H_r \subset \mathsf q (H_{r-1}) \times \mathsf q (F) \,,
\]
be defined as the submonoid generated by
\[
H_{r-1},  u_{r,1} , \ldots ,u_{r, k_r -1}, \ \quad \text{and by } \quad u_{r, k_r} := u_{1,1} \cdot \ldots \cdot u_{1, k_1} \big( u_{r,1} \cdot \ldots  \cdot u_{r, k_r -1 } \big)^{-1} \,.
\]
Then, by construction, $H_r$ is a reduced monoid, which is  generated by $A = \{ u_{i,j} \colon j \in [1, k_i], i \in [1,r] \}$. Obviously,  $A$ is a minimal generating set, whence $A$ is the set of atoms of $H_r$ by \cite[Proposition 1.1.7]{Ge-HK06a}. We continue with three assertions.

\begin{enumerate}
\item[{\bf A1.}\,] $H_r$ is root-closed (i.e., if $x \in \mathsf q (H_r)$ and $m \in \N$ with $x^m \in H_r$, then $x \in H_r$).

\smallskip

\item[{\bf A2.}\,] Every $b \in H_r \setminus u_{1,1} \cdot \ldots \cdot u_{1, k_1}H_r$ has unique factorization.

\smallskip

\item[{\bf A3.}\,] For every $a \in H_r$, there are unique $n \in \N$ and  unique $b \in H_r \setminus u_{1,1} \cdot \ldots \cdot u_{1, k_1}H_r$ such that $\mathsf Z (a) = \mathsf Z \big(  (u_{1,1} \cdot \ldots \cdot u_{1,k_1})^n \big) \mathsf Z (b)$ and
    \[
    \begin{aligned}
    \mathsf Z \big(  (u_{1,1} \cdot \ldots \cdot u_{1,k_1})^n \big) = \big\{ (u_{1,1} \cdot \ldots \cdot u_{1,k_1})^{n_1} & \cdot \ldots \cdot (u_{r,1} \cdot \ldots \cdot u_{r,k_r})^{n_r} \colon \\ &  (n_1, \ldots, n_r) \in \N_0^r \ \text{with} \ n_1+ \ldots + n_r = n \big\} \,.
    \end{aligned}
    \]
\end{enumerate}

\smallskip
We suppose that {\bf A1}, {\bf A2}, and {\bf A3} hold and complete the proof of the proposition. Since finitely generated root-closed monoids are Krull by \cite[Theorem 2.7.14]{Ge-HK06a}, $H_r$ is a Krull monoid by {\bf A1}. Clearly,  Property (a) holds and {\bf A2} is equal to  Property (b). Furthermore, Assertion {\bf A3} implies Property (c) and that, for every $a \in H_r$,
\[
\mathsf c (a) = \begin{cases}
                 1 & \text{if $n=0$}, \\
                 k_r & \text{if $n > 0$}.
                 \end{cases}
\]
Thus $\mathsf c (H_r)=k_r$. Finally, Property (c) implies that $\mathcal L (H_r)$ has the given form.

\bigskip
\noindent
{\it Proof of \,{\bf A1}}.\, Let $x\in \mathsf q(H_r)$ be such that  $x^m\in H_r$ for some $m \in \N$. We have to show that $x \in H_r$. Since $\mathsf q(H_r)\subset \mathsf q(H_{r-1})\times \mathsf q(F)$,
there are $y\in \mathsf q(H_{r-1})$, $s_1, \ldots, s_{k_r-1}\in \Z$, $u\in H_{r-1}$, and $t_1,\ldots, t_{k_r}\in \N_0$ such that
\[
x=yu_{r,1}^{s_1}\ldots u_{r, k_r-1}^{s_{k_r-1}} \ \text{ and }\ x^m=uu_{r,1}^{t_1}\ldots u_{r, k_r}^{t_{k_r}}\,.
\]
Since $u_{r,1}\ldots u_{r,k_r}\in H_{r-1}$, we may assume that either $t_{k_r}=0$ or there exists $i\in [1, k_r-1]$ such that $t_i=0$.

If $t_{k_r}=0$, then $x^m=y^mu_{r,1}^{ms_1}\ldots u_{r, k_r-1}^{ms_{k_r-1}}=uu_{r,1}^{t_1}\ldots u_{r,k_r-1}^{t_{k_r-1}}\in \mathsf q(H_{r-1})\times \mathsf q(F)$ implies that $y^m=u\in H_{r-1}$ and $ms_i=t_i$ for every $i\in [1, k_r-1]$. Therefore $s_i\ge 0$ for all $i\in [1, k_r-1]$. Since $H_{r-1}$ is a reduced finitely generated Krull monoid, it follows by \cite[Theorem 2.7.14]{Ge-HK06a} that $H_{r-1}$ is root-closed, whence $y\in H_{r-1}$. Therefore $x=yu_{r,1}^{s_1}\ldots u_{r,k_r-1}^{s_{k_r-1}}\in H_{r}$.

Suppose there is  $i\in [1, k_r-1]$ such that $t_i=0$, say $i=1$. Then
\begin{align*}
x^m&=y^mu_{r,1}^{ms_1}\ldots u_{r, k_r-1}^{ms_{k_r-1}}=uu_{r,2}^{t_2}\ldots u_{r,k_r}^{t_{k_r}}\\
&=uu_{r,2}^{t_2}\ldots u_{r,k_r-1}^{t_{k_r-1}}(u_{1,1}\ldots u_{1,k_1})^{t_{k_r}}(u_{r,1}\ldots u_{r,k_r-1})^{-t_{k_r}}\\
&=u(u_{1,1}\ldots u_{1,k_1})^{t_{k_r}}\ u_{r,1}^{-t_{k_r}}u_{r,2}^{t_2-t_{k_r}}\ldots u_{r,k_r-1}^{t_{k_r-1}-t_{k_r}}\in \mathsf q(H_{r-1})\times \mathsf q(F)\,,
\end{align*}
which implies that $y^m=u(u_{1,1}\ldots u_{1,k_1})^{t_{k_r}}\in H_{r-1}$ and $ms_i=t_i-t_{k_r}$ for all $i\in [1, k_r-1]$. We set $t_{k_r}=qm+m_0$, where $q\in \N_0$ and $m_0\in [0, m-1]$ and obtain that
$$(y(u_{1,1}\ldots u_{1,k_1})^{-q})^m=y^m(u_{1,1}\ldots u_{1,k_1})^{-t_{k_r}+m_0}=u(u_{1,1}\ldots u_{1,k_1})^{m_0}\in H_{r-1}\,.$$
Since $H_{r-1}$ is a reduced finitely generated Krull monoid, it follows by \cite[Theorem 2.7.14]{Ge-HK06a} that $H_{r-1}$ is root-closed, whence $y(u_{1,1}\ldots u_{1,k_1})^{-q}\in H_{r-1}$ and
\begin{align*}
x&=yu_{r,1}^{s_1}\ldots u_{r,k_r-1}^{s_{k_r-1}}=yu_{r,1}^{\lfloor \frac{t_1}{m}\rfloor-q}\ldots u_{r, k_r-1}^{\lfloor \frac{t_{k_r-1}}{m} \rfloor-q}\\
&=y(u_{1,1}\ldots u_{1,k_1})^{-q}u_{r,k_r}^qu_{r,1}^{\lfloor \frac{t_1}{m}\rfloor}\ldots u_{r, k_r-1}^{\lfloor \frac{t_{k_r-1}}{m} \rfloor}\in H_r\,.
\end{align*}

\smallskip
\noindent
{\it Proof of \,{\bf A2}}.\,  Let $b\in H_r\setminus u_{1,1} \cdot \ldots \cdot u_{1,k_1}H_r$. Factorizations $z_1, z_2 \in \mathsf Z (b)$ can be written in the form
\[
z_1=x\cdot y\cdot u_{r,k_r}^{t} \quad \text{and} \quad z_2=x'\cdot y'\cdot u_{r,k_r}^{t'}\,, \text{ where }x,x'\in \mathsf Z(H_{r-1}), y,y'\in \mathsf Z(F), \text{ and } t,t'\in \N_0\,.
\]
We have to show that $z_1 = z_2$.
By symmetry, we may assume that $t\ge t'$. Since $u_{r,1}\ldots u_{r,k_r}=u_{1,1}\ldots u_{1,k_1}$, we infer that
\[
\pi(x)\pi(x')^{-1}(u_{1,1} \cdot \ldots \cdot u_{1,k_1})^{t-t'}=\pi(y')\pi(y)^{-1}(u_{r,1} \cdot \ldots \cdot u_{r,k_r-1})^{t-t'}\in \mathsf q(H_{r-1})\cap \mathsf q(F)=\{1\}\,,
\]
whence $\pi(x)(u_{1,1} \cdot \ldots \cdot u_{1,k_1})^{t-t'}=\pi(x')$. Since $b=\pi(x')\pi(y')u_{r,k_r}^{t'}$ and $u_{1,1} \cdot \ldots \cdot u_{1,k_1}$ does not divide $b$, it follows that $t=t'$. Thus, we obtain that
$\pi(x)=\pi(x')$ and $\pi(y)=\pi(y')$, whence $y=y'$. Since $x,x'\in \mathsf Z(\pi(x))\subset \mathsf Z(H_{r-1})$, the induction hypothesis implies  $x=x'$,
whence $z_1=z_2$,

\smallskip
\noindent
{\it Proof of \,{\bf A3}}.\, Let $a\in H_r$ and let $n\in \N_0$ be the maximal integer such that $(u_{1,1}\ldots u_{1,k_1})^n$ divides $a$ and set $b=a(u_{1,1}\ldots u_{1,k_1})^{-n}$. Then
$a=(u_{1,1}\ldots u_{1,k_1})^n b$ and hence
 \[
\big\{ (u_{1,1} \cdot \ldots \cdot u_{1,k_1})^{n_1}  \cdot \ldots \cdot (u_{r,1} \cdot \ldots \cdot u_{r,k_r})^{n_r} \colon   (n_1, \ldots, n_r) \in \N_0^r \ \text{with} \ n_1+ \ldots + n_r = n \big\} \cdot \mathsf Z(b)\subset \mathsf Z(a)\,.
\]
Conversely, let $$z=(u_{1,1} \cdot \ldots \cdot u_{1,k_1})^{t_1}  \cdot \ldots \cdot (u_{r,1} \cdot \ldots \cdot u_{r,k_r})^{t_r}\cdot x\cdot y\cdot u_{r,k_r}^t$$ be a factorization of $a$, where $t_1,\ldots, t_r,t\in \N_0$, $y\in \mathsf Z(F)$, and $x\in \mathsf Z(H_{r-1})$ such that $u_{r,1}\cdot\ldots\cdot u_{r,k_r}$ does not divide $y\cdot u_{r,k_r}^t$ in $\mathsf Z (H_r)$ and $u_{i,1}\cdot \ldots \cdot u_{i,k_i}$ does not divide $x$ in $\mathsf Z(H_{r-1})$ for every $i\in [1, r-1]$. If $t_1+\ldots+t_r>n$, then $(u_{1,1}\ldots u_{1,k_1})^{n+1}$ divides $a$, a contradiction to the maximality of $n$. If $t_1+\ldots+t_r=n$, then $$z\in \big\{ (u_{1,1} \cdot \ldots \cdot u_{1,k_1})^{n_1}  \cdot \ldots \cdot (u_{r,1} \cdot \ldots \cdot u_{r,k_r})^{n_r} \colon   (n_1, \ldots, n_r) \in \N_0^r \ \text{with} \ n_1+ \ldots + n_r = n \big\} \cdot \mathsf Z(b)\,.$$
Assume to the contrary that $t_1+\ldots+t_r<n$. Then $(u_{1,1} \cdot \ldots \cdot u_{1,k_1})$ divides $\pi(x\cdot y\cdot u_{r,k_r}^t)$, say $c=\pi(x\cdot y\cdot u_{r,k_r}^t)$. Thus, $c$ has  factorizations
\[
z_1=x\cdot y\cdot u_{r,k_r}^t \ \text{ and } \ z_2=(u_{1,1}\cdot \ldots \cdot u_{1,k_1})\cdot x'\cdot y'\cdot u_{r,k_r}^{t'}\,,
\]
where $x'\in \mathsf Z(H_{r-1})$, $y'\in \mathsf Z(F)$, and $t'\in \N_0$. It follows that
$$\pi(x)\pi(x')^{-1}(u_{1,1} \cdot \ldots \cdot u_{1,k_1})^{t-t'-1}=\pi(y')\pi(y)^{-1}(u_{r,1} \cdot \ldots \cdot u_{r,k_r-1})^{t-t'}\in \mathsf q(H_{r-1})\cap \mathsf q(F)=\{1\}\,.$$
If $t>t'$, then $t>0$ and $\pi(y)=\pi(y')(u_{r,1} \cdot \ldots \cdot u_{r,k_r-1})^{t-t'}$, whence $u_{r,1}\cdot \ldots \cdot u_{r,k_r}$ divides $y\cdot u_{r,k_r}^t$, a contradiction to our assumption on $y\cdot u_{r,k_r}^t$. Therefore $t\le t'$ and $\pi(x)=\pi(x')(u_{1,1} \cdot \ldots \cdot u_{1,k_1})^{1+t'-t}$. Let $m\in \N$ be the maximal integer such that $(u_{1,1} \cdot \ldots \cdot u_{1,k_1})^{m}$ divides $\pi(x)$ and let $c_0=\pi(x)(u_{1,1} \cdot \ldots \cdot u_{1,k_1})^{-m}$. By the induction hypothesis, $x$ is in
\[
\big\{ (u_{1,1} \cdot \ldots \cdot u_{1,k_1})^{n_1}  \cdot \ldots \cdot (u_{r-1,1} \cdot \ldots \cdot u_{r-1,k_{r-1}})^{n_{r-1}} \colon   (n_1, \ldots, n_{r-1}) \in \N_0^{r-1} \ \text{with} \ n_1+ \ldots + n_{r-1} = m \big\} \cdot \mathsf Z_{H_{r-1}}(c_0)\,,
\]
 a contradiction to our assumption on $x$. Thus, we obtained that
 \[
 \mathsf Z(a)=\big\{ (u_{1,1} \cdot \ldots \cdot u_{1,k_1})^{n_1}  \cdot \ldots \cdot (u_{r,1} \cdot \ldots \cdot u_{r,k_r})^{n_r} \colon   (n_1, \ldots, n_r) \in \N_0^r \ \text{with} \ n_1+ \ldots + n_r = n \big\} \cdot \mathsf Z(b) \,. \qedhere
 \]
\end{proof}

\smallskip
A monoid homomorphism $\theta \colon H \to B$ between atomic monoids is said to be a {\it transfer homomorphism} if the following two properties are satisfied.
\begin{enumerate}
\item[{\bf (T\,1)\,}] $B = \theta(H) B^\times$ \ and \ $\theta
^{-1} (B^\times) = H^\times$.

\item[{\bf (T\,2)\,}] If $u \in H$, \ $b,\,c \in B$ \ and \ $\theta
(u) = bc$, then there exist \ $v,\,w \in H$ \ such that \ $u = vw$, \
$\theta (v) \in b B^{\times}$ \ and \ $\theta (w) \in c B^{\times}$.
\end{enumerate}
One of the main properties of  transfer homomorphisms is that they preserve sets of lengths. That is, if $\theta \colon H \to  B $ is a  transfer homomorphism, then $\mathsf L_H (a) = \mathsf L_{B } ( \theta (a))$ for all $a \in H$, whence
\begin{equation} \label{equality-of-systems}
 \mathcal L (H) = \mathcal L ( B ) \,.
\end{equation}
A monoid is said to be a {\it transfer Krull monoid} (of finite type) if there is a  transfer homomorphism $\theta \colon H \to \mathcal B (G_0)$ for a (finite) subset $G_0$ of an abelian group $G$. If $H$ is a Krull monoid with divisor theory $\partial \colon H \to \mathcal F (P)$ and class group $G=\mathcal C (H)$, then there is a transfer homomorphism
\begin{equation} \label{Krull-transfer}
\theta \colon H \to \mathcal B ( G_P) \,,
\end{equation}
where $G_P = \{ [p] \colon p \in P \} \subset G$ is the set of classes containing prime divisors (\cite[Theorem 3.4.10]{Ge-HK06a}).
However, the concept of transfer Krull monoids is neither restricted to the commutative nor to the cancellative setting (we refer to the survey \cite{Ge-Zh20a}; for non-commutative transfer Krull domains see \cite[Section 7]{Ba-Sm15} and \cite[Theorem 4.4]{Sm19a};  a commutative, but non-cancellative semigroup of modules over Bass rings, that is transfer Krull, is studied in \cite{Ba-Sm21}).

The next lemma (whose proof is straightforward) reveals that families $\mathcal L^*$, as in Theorem \ref{1.1}, cannot only be realized as systems of sets of lengths of Krull monoids, but also by wide classes of transfer Krull monoids. We will use Lemma \ref{3.2} in the proof of Corollary \ref{1.2}.

\medskip
\begin{lemma} \label{3.2}
Let $\mathcal L^*$ be a family of finite subsets of $\N_0$ satisfying Properties (a) and (b), given in Theorem \ref{1.1}. Let $H$ be a Krull monoid with transfer homomorphism as in \eqref{Krull-transfer} such that $\mathcal L (H) = \mathcal L^*$. If $H^*$ is a transfer Krull monoid with transfer homomorphism $\theta^* \colon H^* \to \mathcal B (G_P)$, then $\mathcal L (H^*) = \mathcal L^*$.
\end{lemma}

\begin{proof}
Applying Equation \eqref{equality-of-systems} twice, we obtain that
\[
\mathcal L^* = \mathcal L (H) = \mathcal L \big( \mathcal B (G_P) \big) = \mathcal L (H^*) \,. \qedhere
\]
\end{proof}

\medskip
\begin{proposition} \label{3.3}
Let $H$ be a  monoid.
\begin{enumerate}
\item If $H_{\red}$ is finitely generated, then there exist $a_1^*, \ldots, a_m^* \in H \setminus H^{\times}$ such that for every $a \in H$ there is $i \in [1, m]$ such that
$\mathsf L (a) = \mathsf L (a_i^*) +  \mathsf L \big( (a_i^*)^{-1} a \big)$.

\item Let $\theta \colon H \to \mathcal B (G_0)$ be a transfer homomorphism, where $G_0$ is a finite subset of an abelian group. Then there exist $a_1^*, \ldots, a_m^* \in H \setminus H^{\times}$ with the following property: for every $a \in H$ there are $i \in [1, m]$ and $a_i' \in H$ such that $a_i' \mid a$, $\theta (a_i') = \theta (a_i^*)$, and $\mathsf L (a) = \mathsf L (a_i^*) +  \mathsf L \big( (a_i')^{-1} a \big)$.
\end{enumerate}
Moreover, if $\mathcal L (H)$ is additively closed, then (both, in 1. as well as in 2.) $\mathcal L (H)$ is a finitely generated semigroup.
\end{proposition}

\begin{proof}
1. Without restriction we may suppose that $H$ is reduced and finitely generated. We define $$S^*=\{a\in H\colon \text{ for any }b\in H \text{ with }b\t a \text{ and }\mathsf L(b)\neq \mathsf L(a),\, \text{we have } \mathsf L(b)+\mathsf L(ab^{-1}) \subsetneq \mathsf L (a) \}\,,$$
and observe  that $\mathcal A(H)\subset S^*$.
Suppose that $S^*$ is finite. We show that for every $a\in H$, there exists $a^*\in S^*$ such that $\mathsf L (a) = \mathsf L (a^*) +  \mathsf L \big( (a_i^*)^{-1} a \big)$. If $a\in S^*$, then it is trivial.
Suppose $a\in H\setminus S^*$. Then there exists $b_1\in H$ with $b_1\t a$ and $\mathsf L(b_1)\neq \mathsf L(a)$ such that $\mathsf L(a)=\mathsf L(b_1)+\mathsf L(ab_1^{-1})$.
Since $\mathsf L(a)$ is finite, there are $k \in \N$ and $b_1,b_2,\ldots, b_k\in H$ with $b_{i+1}\t b_i$ for all $i\in [1,k-1]$ such that $b_k\in S^*$ and $$\mathsf L(a)=\mathsf L(b_k)+\mathsf L(b_{k-1}b_k^{-1})+\ldots+\mathsf L(b_1b_2^{-1})+\mathsf L(ab_1^{-1})\subset \mathsf L(b_k)+\mathsf L(ab_k^{-1})\subset \mathsf L(a)\,,$$
whence the assertion follows.

Assume to the contrary that $S^*$ is infinite. The set of distances of finitely generated monoids is finite and their sets of lengths are well-structured. Indeed, there is a bound $M \in \N$ such that for every $a \in H$ there are $d \in \Delta (H)$ and a set $\mathcal D$ with $\{0,d\} \subset \mathcal D \subset [0,d]$ such that
\begin{equation} \label{structure}
\mathsf L (a) = y + \big(L' \cup L^* \cup (\max L^* + L'') \big) \subset y + \mathcal D + d \Z \,,
\end{equation}
where $y \in Z$, $L' \subset [-M, -1]$,  $L'' \subset [1,M]$, and $L^* = \mathcal D +d\cdot [0,\ell]$ for some $\ell \in \N_0$ (this is the "in particular" statement of \cite[Theorem 4.4.11]{Ge-HK06a}).
Since $\Delta (H)$ is finite,
 there exist  $d\in \Delta (H)$, $\{0,d\}\subset\mathcal D\subset [0,d]$, $L'\subset [-M,-1]\cap d\Z$, and  $L''\subset [1,M]\cap d\Z$ such that for   an infinite subset $S^{**}\subset S^*$ and all $a \in S^{**}$
 \begin{equation}\label{eq-1}
\mathsf L(a) = y + (L' \cup L^*\cup (\max L^*+L'')) \subset y + \mathcal D + d \Z \,,
 \end{equation}
where  $\ell\in \N_{\ge \max\Delta(H)}$ and all other parameters as in \eqref{structure}.

We define $H_0=\{(x,y)\in \mathsf Z(H)\times \mathsf Z(H)\colon \pi(x)=\pi(y)\}$ and observe that $H_0$ is a saturated submonoid of the finitely generated monoid $\mathsf Z(H)\times \mathsf Z(H)$, whence  $H_0$ is a finitely generated monoid by \cite[Proposition 2.7.5]{Ge-HK06a}.  We set $\mathcal A(H_0)=\{(x_i,y_i)\colon i\in [1,t]\}$ with $t\in \N$.

For every $a\in S^{**}$, we let $(z_a,w_a)\in H_0$ with $\pi(z_a)=a$, $|z_a|=\max\mathsf L(a)$, and $|w_a|=\min\mathsf L(a)$. Then there exist $k_{a,1},\ldots, k_{a,t}\in \N_0$ such that $(z_a,w_a)=\prod_{i=1}^t(x_i,y_i)^{k_{a,i}}$, whence
\[
\max \mathsf L(a)=|z_a|=\sum_{i=1}^tk_{a,i}|x_i|=\sum_{i=1}^tk_{a,i}\max\mathsf L(\pi(x_i))
\]
and
\[
\min \mathsf L(a)=|w_a|=\sum_{i=1}^tk_{a,i}|y_i|=\sum_{i=1}^tk_{a,i}\min\mathsf L(\pi(x_i))\,.
\]
Since $S^{**}$ is  infinite, there exist  $b,c\in S^{**}$ with $\mathsf L(b)\neq \mathsf L(c)$ such that $k_{b,i}\le k_{c,i}$ for all $i\in [1,t]$. Then $b$ divides $c$ and so $\mathsf L (b ) + \mathsf L (cb^{-1}) \subset \mathsf L (c)$, whence
\begin{align*}
\max\mathsf L(b)+\max\mathsf L(cb^{-1})&\le \max\mathsf L(c)=\sum_{i=1}^tk_{c,i}|x_i|\le  \max\mathsf L(b)+\sum_{i=1}^t(k_{c,i}-k_{b,i})|x_i| \\
  & \le \max\mathsf L(b)+\max\mathsf L(cb^{-1}) \\
\text{and} & \\
\min\mathsf L(b)+\min\mathsf L(cb^{-1})&\ge \min\mathsf L(c)=\sum_{i=1}^tk_{c,i}|y_i|\ge  \min\mathsf L(b)+\sum_{i=1}^t(k_{c,i}-k_{b,i})|y_i| \\ & \ge \min\mathsf L(b)+\min\mathsf L(cb^{-1})\,.
\end{align*}
Therefore, we obtain
\begin{equation}\label{eq-2}
\max\mathsf L(c)=\max\mathsf L(b)+\max\mathsf L(cb^{-1})
\ \text{ and }\ \min\mathsf L(c)=\min\mathsf L(b)+\min\mathsf L(cb^{-1})\,.
\end{equation}
In view of \eqref{eq-1}, we have $$\mathsf L(b)=y_b+(L'\cup L_b\cup (\max L_b+L'')) \subset y_b + \mathcal D + d \Z $$ and $$\mathsf L(c)=y_c+(L'\cup L_c\cup (\max L_c+L'')) \subset y_c + \mathcal D + d \Z \,,$$ where $y_b,y_c\in \N_0$, $L_b=\mathcal D+d\cdot [0,\ell_b]$, $L_c=\mathcal D+d\cdot[0,\ell_c]$, and $\ell_b,\ell_c\in \N_{\ge \max\Delta(H)}$.
It follows by \eqref{eq-2} that $\{y_c-y_b, y_c-y_b+\max L_c-\max L_b\}\subset \mathsf L(cb^{-1})$. Therefore
\begin{equation}\label{eq-3}
	\begin{aligned}
	&	y_c+L'=y_c-y_b+y_b+L'\subset \mathsf L(cb^{-1})+\mathsf L(b) \text{ and}\\
	&	 y_c+\max L_c+L''=(y_c-y_b+\max L_c-\max L_b)+(y_b+\max L_b+L'')\subset \mathsf L(cb^{-1})+\mathsf L(b)\,.	
		\end{aligned}
	\end{equation}
Suppose $\mathsf L(cb^{-1})=\{y_c-y_b=n_1, n_2,\ldots, n_{|L_0|}=y_c-y_b+\max L_c-\max L_b\}$ with $n_i<n_j$ if $i<j$. Then
\[
\bigcup_{i=1}^{|L_0|}n_i+\mathcal D+d\cdot[0, \ell_b] = \mathsf L(cb^{-1})+L_b \subset \mathsf L (cb^{-1}) + (-y_b + \mathsf L (b)) \subset -y_b + \mathsf L (c) \subset (y_c-y_b)+ \mathcal D+d\Z \,.
\]
Since $\ell_b\ge \max\Delta(H)$, we have $n_i+\max L_b\ge n_{i+1}+\min L_b$ for every $i\in [0, |L_0|-1]$. It follows that $\mathsf L(cb^{-1})+L_b=y_c-y_b+L_c$ (because $y_c-y_b+\max L_c-\max L_b = \max \mathsf L (cb^{-1})$), whence
\begin{equation}
y_c + L_c = y_b +L_b + \mathsf L (cb^{-1}) \subset \mathsf L (b) + \mathsf L (cb^{-1}) \,.
\end{equation}
By this inclusion together with \eqref{eq-3}, we obtain that $\mathsf L(c)\subset \mathsf L(b)+\mathsf L(cb^{-1})$ and hence $\mathsf L(c)=\mathsf L(b)+\mathsf L(cb^{-1})$, a contradiction to  $c \in S^*$.

\smallskip
2.  Since $G_0$ is finite, $\mathcal B (G_0)$ is finitely generated by \cite[Theorem 3.4.2]{Ge-HK06a}. Thus 1. implies that there are $A_1^* , \ldots, A_m^* \in \mathcal B (G_0)$ such that for every $A \in \mathcal B (G_0)$ we have $\mathsf L (A) = \mathsf L (A_i^*) +  \mathsf L \big( (A_i^*)^{-1} A \big)$. We choose $a_1^*, \ldots, a_m^* \in H$ with $\theta (a_i^*) =A_i^*$. Let $a \in H$. Then there is $i \in [1,m]$ such that $\mathsf L ( \theta (a) ) = \mathsf L (A_i^*) +  \mathsf L \big( (A_i^*)^{-1} \theta (a) \big)$. By {\bf T2}, there are $a_i', b \in H$ such that $a = a_i'b$, $\theta (a_i') = A_i^*$, and $\theta (b) = (A_i^*)^{-1} \theta (a)$. Thus, we obtain that
\[
\mathsf L (a) = \mathsf L \big( \theta (a) \big) = \mathsf L (A_i^*) +  \mathsf L \big( (A_i^*)^{-1} \theta (a) \big) = \mathsf L (a_i') + \mathsf L (b) \,.
\]

Suppose that  $\mathcal L (H)$ is additively closed. Then $\mathcal L (H)$ is a semigroup with set addition as operation, and   $\{ \mathsf L (a_1^*), \ldots, \mathsf L (a_m^*) \}$ is a generating set of $\mathcal L (H)$.
\end{proof}

\bigskip
\begin{proof}[Proof of Theorem \ref{1.1}]
Let $\mathcal L^*$ be a family of finite subsets of $\N_0$ such that
\begin{itemize}
\item[(a)] $\{0\}, \{1\} \in \mathcal L^*$ and all other sets of $\mathcal L^*$ lie in $\N_{\ge 2}$.

\item[(b)] If $L_1, L_2 \in \mathcal L^*$, then $L_1 + L_2 \in \mathcal L^*$.
\end{itemize}
This means that $\mathcal L^*$ is a commutative semigroup with set addition as operation and with $\{0\}$ being the identity element. Let $A = \{A_i \colon i \in I \}$ be the set of indecomposable elements of $\mathcal L^*$. Proposition \ref{3.1} implies that, for every $i \in I$, there is a finitely generated Krull monoid $H_i$ such that
\[
\mathcal L (H_i) = \big\{ \{0\}, \{1\} \big\} \cup \{ y + n A_i \colon y, n \in \N_0 \} \,.
\]
We set $H = \coprod_{i \in I} H_i$ and note that
\[
\mathcal L (H) = \big\{ \sum_{i \in I} L_i \colon L_i \in \mathcal L (H_i) \ \text{and all but finitely many $L_i$ are equal to $\{0\}$} \big\} \,.
\]
Since $\mathcal L^*$ is a semigroup and $A$ is the   set of atoms, we infer that $\mathcal L^* = \mathcal L (H)$.
If $A$ is finite, then $H$ is finitely generated because all $H_i$ are finitely generated. Conversely, suppose that $H$ is finitely generated. Then, by \eqref{Krull-transfer} and Lemma \ref{2.1},  there is a transfer homomorphism $\theta \colon H \to \mathcal B (G_P)$, where $G_P \subset \mathcal C (H)$ is  finite. Thus Proposition \ref{3.3} implies that $\mathcal L (H) = \mathcal L^*$ has only finitely many indecomposable sets.
\end{proof}

\bigskip
\begin{proof}[Proof of Corollary \ref{1.2}]
Let $\mathcal L^*$ be a family of finite subsets of $\N_0$ satisfying Properties (a) and (b), given in Theorem \ref{1.1}. Now, by Theorem \ref{1.1}, there is a Krull monoid $H$ with divisor theory $\partial \colon H \to \mathcal F (P)$ such that $\mathcal L^* = \mathcal L (H)$, and we may suppose that $H$ is reduced. By \eqref{Krull-transfer}, there is a transfer homomorphism $\theta \colon H \to \mathcal B (G_P)$, where $G_P \subset \mathcal C (H)$ is the set of classes  containing prime divisors. By Claborn's Realization Theorem (\cite[Theorem 3.7.8]{Ge-HK06a}), there is a Dedekind domain $D$ and an isomorphism
\[
\Phi \colon G \to \mathcal C (D) \quad \text{such that} \quad \Phi (G_P) = \{ g \in \mathcal C (D) \colon \mathfrak X (D) \cap g \ne \emptyset \} \,,
\]
where $\mathcal C (D)$ is the class group of $D$ and $\mathfrak X (D)$ is the set of height-one prime ideals of $D$. Since $D$ is a Krull domain and $D^{\bullet} = D \setminus \{0\}$ is a Krull monoid, we have, again by \eqref{Krull-transfer}, a transfer homomorphism $\theta^* \colon D^{\bullet} \to \mathcal B \big( \Phi (G_P) \big) \cong \mathcal B (G_P)$. Thus, Lemma \ref{3.2} implies that $\mathcal L^* =  \mathcal L (D)$.

If $\mathcal L^*$ has only finitely many indecomposable sets, the  monoid $H$ is finitely generated by Theorem \ref{1.1}, whence the set of prime divisors $P$ is finitely generated by Lemma \ref{2.1}. Thus $G_P$ and $\Phi (G_P)$ are finite.

Conversely, suppose that there is a Dedekind domain $D^*$ with $\mathcal L (D^*) = \mathcal L^*$ and whose set of classes $G_0^* \subset \mathcal C (D^*)$ containing height-one prime ideals is finite. Then $\mathcal B (G_0^*)$ is finitely generated and $\mathcal L^* = \mathcal L (D^*) = \mathcal L \big( \mathcal B (G_0^*) \big)$, whence $\mathcal L^*$ has only finitely many indecomposable sets by Proposition \ref{3.3}.
\end{proof}

\medskip
\noindent
{\bf Acknowledgement} We thank the referees for the careful reading. Their comments helped to improve the presentation of the paper.

\providecommand{\bysame}{\leavevmode\hbox to3em{\hrulefill}\thinspace}
\providecommand{\MR}{\relax\ifhmode\unskip\space\fi MR }
\providecommand{\MRhref}[2]{%
  \href{http://www.ams.org/mathscinet-getitem?mr=#1}{#2}
}
\providecommand{\href}[2]{#2}

\end{document}